\documentclass{article}
\usepackage{amsmath}
  \usepackage{paralist}
  \usepackage{graphics} %% add this and next lines if pictures should be in esp format
  \usepackage{epsfig} %For pictures: screened artwork should be set up with an 85 or 100 line screen
\usepackage{amsthm}
\usepackage{amssymb}
\usepackage{bbm}
\usepackage{nicefrac}
\usepackage{graphicx}  \usepackage{epstopdf}%This is to transfer .eps figure to .pdf figure; please compile your paper using PDFLeTex or PDFTeXify.
 \usepackage[colorlinks=true]{hyperref}

 \usepackage{tikz}
 \usetikzlibrary{graphs,graphs.standard,quotes}
\usepackage{subcaption}

 \usepackage{cleveref}
 \usepackage{float}

\hypersetup{urlcolor=blue, citecolor=red}

\newtheorem{Theorem}{Theorem}[section]

\newtheorem{Lemma}[Theorem]{Lemma}
\newtheorem{Proposition}[Theorem]{Proposition}

\theoremstyle{definition}
\newtheorem{definition}[Theorem]{Definition}
\newtheorem{Remark}[Theorem]{Remark}
\newtheorem{Example}[Theorem]{Example}

\def\nfrac#1#2{{\textstyle\frac{#1}{#2}}}

\title{Consensus and Voting on Large Graphs:\\ An Application of Graph Limit Theory}

\begin{document}
\maketitle

% Enter the first author's name and address:
\centerline{\scshape Barton E. Lee\footnote{Email: barton.e.lee@gmail.com}}
\medskip
{\footnotesize
% please put the address of the first author
 \centerline{ School of Mathematics and Statistics}
   \centerline{The University of New South Wales}
   \centerline{ Sydney NSW 2052, Australia}
} % Do not forget to end the {\footnotesize by the sign }

\bigskip

% The name of the associate editor will be entered by an editorial staff
% "Communicated by the associate editor name" is not needed for special issue.
% \centerline{(Communicated by the associate editor name)}

%The abstract of your paper
\begin{abstract}
Building on recent work by Medvedev (2014) we establish new connections between a basic consensus model, called the voting model, and the theory of graph limits. We show that in the voting model if consensus is attained in the continuum limit then solutions to the finite model will eventually be close to a constant function, and a class of graph limits which guarantee consensus is identified. It is also proven that the dynamics in the continuum limit can be decomposed as a direct sum of dynamics on the connected components, using Janson's definition of connectivity for graph limits. This implies that without loss of generality it may be assumed that the continuum voting model occurs on a connected graph limit.
\end{abstract}

%The title of your section 1

\section{Introduction}\label{introsec}

Given a group of agents voting on whether or not to implement a policy, when and how can the group come to an agreement? This is a consensus problem. Models and algorithms which solve consensus problems are important in the theory of control systems \cite[Section 16.2]{Lun} and economics \cite{Jac1}. In this paper we focus on a particular model of the consensus problem called the \emph{voting model}. This model fits into the framework considered by Medvedev \cite{Med}, who applied graph limit theory to sequences of dynamical systems.

Our key result shows that if the solution to the continuum limit of the voting model reaches consensus then solutions to the finite model on sufficiently large graphs will eventually be close to a constant function. Furthermore, building on our concept of twin-kernels a class of graph limits which guarantee consensus is identified. We also prove that the dynamics in the continuum limit can be decomposed as a direct sum of the dynamics on connected components. This means that without loss of generality we may assume that the dynamics of the continuum model occur on a connected graph limit. These results provide motivation for the continued study of graph limit theory in the consensus protocol literature.

The structure of the remainder of this paper is as follows. In \Cref{s2} we will give a brief introduction to the voting model in the finite setting and \Cref{s3} will review some basic results from the graph limit theory literature. \Cref{s4} extends the model to the continuum setting and \Cref{s6} studies consensus in the continuum model and its relationship with the finite model. \Cref{s6.1} focuses on a special case of the continuum voting model and classifies a class of graph limits which guarantee consensus. The paper concludes with \Cref{srand} which extends the analysis to the context of random graphs.

Some of the results presented in this paper such as \Cref{decomposedynamics} extend to the more general class of nonlinear heat equation initial value problems considered by Medvedev in \cite{Med}. Other results such as \Cref{keytheorem}, \Cref{consisconstant} and \Cref{keytheorem1} rely on a conservation condition (\Cref{conservation}) which holds when an additional condition on the heat equation is satisfied. The remaining results such as those in \Cref{s6.1} rely on specific aspects of the voting model we consider.  However, in any case the methods used throughout this paper may prove useful for researchers interested in initial value problems on large graphs and results about consensus.

\subsection{Limitations and scope}\label{limitation}

The main contribution of this paper is the application of graph limits to approximately solve a consensus problem (\Cref{keytheorem}) and the methods developed within. The practical application, however, is limited by the difficulty involved in finding solutions to the continuum model which attain consensus for interesting graph sequences.  

\Cref{s6.1} absolves this difficulty for a class of graph sequences by showing that consensus is always attained in such cases. Applying our key result then shows that consensus can be approximately guaranteed for large graphs within such sequences. However, a stronger result can be attained since the Laplacian matrix of a connected graph with nonnegative weights will always attain consensus \cite[Theorem 1]{Saber1}. This has meant that our examples illustrating results developed in \Cref{s6.1} may be solved by other methods. It is important to note that despite this drawback our results are not trivialised - the key result applies to any graph sequence which attains consensus in the limit, only the class identified in \Cref{s6.1} is affected.

%%%%%%%%%%%%%%%%%%%

\section{Finite voting model}\label{s2}

%%%%%%%%%%%%%%%%%%%%%%%%%%%%%%%%%%%%%%%%%%%%%%%%%%

Various voting models have been formulated in the consensus protocol literature \cite{Ald, Dyer}. We consider an elementary form which is also studied in \cite{Med2} and \cite{Sab1}. However, the focus of this paper differs from the existing literature by considering the voter model on sequences of growing graphs; that is, sequences of graphs with vertex sets whose size increases unboundedly. Particular emphasis is placed on approximating the long term behaviour of the voter model on large graphs. \Cref{s4} will extend the voting model to the continuum setting which will assist in determining the long term behaviour of the model on sufficiently large graphs. 

We first begin with some basic definitions from graph theory:

A graph is an ordered pair of sets, say $G=(V,E)$, with $V$ denoting the vertex set and $E$ denoting the edge set of the graph. The elements of $E$ are two element subsets of $V$. This definition implies that the graph is simple; that is, the edge set contains no loops or multiple edges. For more information on graphs we refer the reader to \cite{Die}.

 An \emph{edge-weighted graph} is a graph $H=(V,E)$ together with a sequence $\{\beta_{ij}\}_{i,j\in V}$ of real \emph{edge weights}, such that $\beta_{ij}=\beta_{ji}$ for all $i,j\in V$ and if $\beta_{ij}\neq 0$ then $\{i,\,j\}$ is an edge of $H$. A simple graph is a special case of an edge-weighted graph where edge weights are 0-1 valued with $\beta_{ii}=0$ for all $i\in V$. Without loss of generality we will only consider edge weights which are contained in the interval $[-1,1]$.

 Given an edge-weighted graph $G_n$ on vertex set $[n]:=\{1,\,2,\,\ldots, \, n\}$ with edge weights $\big\{\beta_{ij}^{(n)}\big\}_{i,j\in[n]} $, our voting model arises from the following process. Begin with a set of voters $[n]$, each with initial opinions denoted by $u_i^{(n)}(0)\in \mathbb{R}$ for $i\in [n]$. The influence of voter $i$ on voter $j$ ($i,j\in [n]$) is represented by the edge weight $\beta_{ij}^{(n)}$. At each infinitesimal time step, every voter $i\in [n]$ updates their opinion based on the average of every other voter $j\in [n]$ - scaled according to $\beta_{ij}^{(n)}$. That is, 
\begin{align}
\label{changeinop}
\frac{d\, u_i^{(n)}(t)}{d \,t}=\frac{1}{n}\sum_{j=1}^n\beta_{ij}^{(n)}\Big(u_j^{(n)}(t)-u_i^{(n)}(t)\Big).
\end{align} 

For a given $j\in [n]$, if $\beta_{ij}^{(n)}>0$ then holding all else equal (\ref{changeinop}) implies that voter $i$ adopts an opinion closer to voter $j$'s opinion in the next time step. Whilst, if $\beta_{ij}^{(n)}<0$ then holding all else equal voter $i$'s opinion will diverge from voter $j$'s opinion in the next time step. 

In \Cref{srand} we will extend our analysis to consider the voting model on sequences of random simple graphs. This allows our results to be applied to the voter model on many well-studied random graph processes such as the Watts-Stogatz small world graph.

It should be noted that the voting model process in this paper represents more general systems than simply voters with opinions. In particular, the aforementioned process has been used in autonomous vehicle control systems \cite{Wei}, to model the spread of alcohol abuse \cite{Fren} and neuronal-network activity \cite{Med4}.

We will exclusively consider time to be continuous with $t\in\mathbb{R}^{\ge0}$. Thus the voting process described above can be expressed as an initial value problem (IVP). Vectors and vector-valued functions will be written in bold font whilst components will not.

For any positive integer $n$, let $\boldsymbol{g}^{(n)} \in \mathbb{R}^n$ and let $H_n$ be an edge-weighted graph with vertex set $[n]$ and edge weights $\{\beta_{ij}\}_{i,j\in [n]}$. Then the evolution of voters' opinions is described by the solution $\boldsymbol{u}^{(n)}: \mathbb{R}^{\ge 0}\rightarrow \mathbb{R}^n$ of the IVP
\begin{align}
\label{IVP}
\begin{cases}
\frac{du_i^{(n)}(t)}{dt}&=\frac{1}{n}\sum_{j=1}^n \beta_{ij}^{(n)}\Big(u_j^{(n)}(t)-u_i^{(n)}(t)\Big)\qquad \text{for all }\, i\in [n] \text{ and } t\in \mathbb{R}^{>0},\\
\boldsymbol{u}^{(n)}(0)&=\boldsymbol{g}^{(n)}.
\end{cases}
\end{align}

We now define consensus, which will be the key focus of this paper from \Cref{s6} onwards.

\begin{definition}
\label{consensus}
For any positive integer $n$, let $\boldsymbol{g}^{(n)}\in \mathbb{R}^n$ and let $H_n$ be an edge-weighted graph on vertex set $[n]$. If $\boldsymbol{u}^{(n)}$ is a solution to (\ref{IVP}) such that 
\begin{align}
\label{consensuseq}
\lim_{t\rightarrow \infty}\,\max_{i,j\in [n]}|u_i^{(n)}(t)-u_j^{(n)}(t)|=0,
\end{align}
then we say \em consensus \em is attained by $\boldsymbol{u}^{(n)}$.
\end{definition}

\begin{Example}
\label{completeexample}
For any positive integer $n$, let $H_n$ be an edge-weighted graph on the vertex set $[n]$ such that the edge weights are all equal to $1$. Then $H_n$ will attain consensus for every initial condition vector $\boldsymbol{g}^{(n)}\in \mathbb{R}^n$. Whilst, if the edge weights are all equal to $-1$ then $\boldsymbol{u}^{(n)}$ reaches consensus if and only if $\boldsymbol{g}^{(n)}$ is a constant vector.\hfill $\diamond$
\end{Example}

The IVP (\ref{IVP}) can be expressed as a linear system of differential equations. Let $B^{(n)}$ be an $n\times n$ matrix with $B_{ij}=\beta_{ij}^{(n)}$ for each $i,j\in [n]$ and let\linebreak $\overline{\beta_i}^{(n)}=\sum_{j=1}^n \beta_{ij}^{(n)}$ for each $i\in [n]$. By considering the matrix \linebreak $D^{(n)}=\frac{1}{n}\Big(B^{(n)}-\mathrm{diag}(\overline{\beta_1}^{(n)},\, \overline{\beta_2}^{(n)},\,\ldots,\,\overline{\beta_n}^{(n)} )\Big)$ we can equivalently write (\ref{IVP}) as 
\begin{align}
\label{altIVP}
\begin{cases}
\frac{d\boldsymbol{u}^{(n)}(t)}{dt}&=D^{(n)}\,\boldsymbol{u}^{(n)}(t),\\
\boldsymbol{u}^{(n)}(0)&=\boldsymbol{g}^{(n)}.
\end{cases}
\end{align}
Note that the matrix $D^{(n)}$ is simply the Laplacian matrix of the graph $H_n$ scaled by $-\frac{1}{n}$. The IVP (\ref{altIVP}) has a unique solution
\begin{align}
\label{altIVPsol}
\boldsymbol{u}^{(n)}(t)=e^{D^{(n)}\,t}\,\boldsymbol{g}^{(n)},
\end{align}
so questions of interest such as whether the agents reach a consensus and the time taken to arrive at this consensus are completely determined via the eigenvalues of the (real, symmetric) matrix $D^{(n)}$. In particular, if the eigenvectors associated with the zero eigenvalue are contained within the subspace spanned by $(1, \, 1, \ldots, \, 1)\in \mathbb{R}^n$ and all other eigenvectors of $D^{(n)}$ have negative eigenvalues then consensus will be reached. If all of the eigenvalues are non-positive then further analysis is required to determine whether or not consensus will be reached.

The problem with this approach is that graphs with billions of vertices are becoming increasingly common in both practice and research. However, existing computational methods for calculating eigenvalues and eigenvectors of matrices do not scale well \cite{Kang}. 

The approach we propose considers sequences of the voter model on graphs which grow unboundedly. When time is continuous, this is equivalent to a sequence of IVPs such as (\ref{altIVP}) for $n\in \mathbb{N}$. By studying the sequence of growing graphs we can under certain conditions construct a limiting object (known as the graph limit) which approximately solves the consensus problem on sufficiently large graphs.

To motivate our approach we introduce an example of the consensus problem which will be approximately solved using the results developed within this paper; that is, without the need to calculate eigenvalues of any matrix. In \Cref{s6} we will refer back to this example to illustrates our key result.

\begin{Example}\label{Example1}
Consider a sequence of edge-weighted graphs $G_n$ given by the following process: define the function $W: [0,1]^2\rightarrow \{-1, \, 1\}$
\begin{align*}
W(x,y)&=\begin{cases}
-1 &\text{if $x,y<\frac{1}{3}$}\\
+1 &\text{otherwise.}
\end{cases}
\end{align*}
Then for each $n\in \mathbb{N}$ let $G_n$ be an edge-weighted graph on the vertex set $[n]$ such that the edge-weight between vertices $(i,\, j)\in [n]^2$ is 
$$\beta_{ij}^{(n)}=n^2\,\int_{\frac{i-1}{n}}^{\frac{i}{n}}\int_{\frac{j-1}{n}}^{\frac{j}{n}} W(x,y) \, dx\, dy.$$
This process of constructing a graph sequence from a function such as $W$ will be revisited in \Cref{approxgraphonforivp}. 

\begin{figure}[H]
\centering
  \includegraphics[width=12cm,height=8cm] {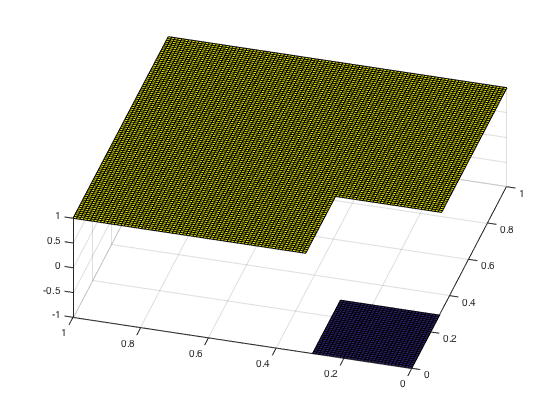}
\caption{A plot of the function $W$.}
\label{fig1}
\end{figure}

\Cref{fig1} contains a plot of the function $W$ and the graphs $G_5$ and $G_6$ are illustrated below. Note that for visual clarity we have omitted the labels for all edge-weights which have value $+1$. 

\begin{figure}[H]
\centering
\begin{subfigure}{0.5\textwidth}
  \centering
  \begin{tikzpicture}[every loop/.style={min distance=8mm,looseness=8}]
  \graph[circular placement,group polar shift=(360/5:0), radius=1.8cm,
         empty nodes, nodes={circle,draw,inner sep=0.2cm}] {
    \foreach \x in {1,2,3,4,5} {
      \foreach \y in {\x,...,5} {
        \x -- \y;
      };
    };

1--[loop above, "-1"]1 ;
2--[loop left, "+1/9"]2;
3--[loop below]3;
4--[loop below]4;
5--[loop right]5;

  };
  \foreach \x [count=\idx from 0] in {$v_1$, $v_2$, $v_3$, $v_4$, $v_5$} {
    \pgfmathparse{90+ \idx * (360 / 5)}
    \node at (\pgfmathresult:1.8cm) {\x};
  };
  
  \node [left] at (-0.85,1.35) {-1/3};

\end{tikzpicture}
\end{subfigure}%
\begin{subfigure}{0.5\textwidth}
  \centering
\begin{tikzpicture}[every loop/.style={min distance=8mm,looseness=8}]
  \graph[circular placement, radius=1.8cm,
         empty nodes, nodes={circle,draw,inner sep=0.2cm}] {
    \foreach \x in {1,2,3,4,5, 6} {
      \foreach \y in {\x,...,6} {
        \x -- \y;
      };
    };

1--[loop above, "-1"]1 ;
2--[loop left, "-1"]2;
3--[loop left]3;
4--[loop below]4;
5--[loop right]5;
6--[loop right]6;

  };
  \foreach \x [count=\idx from 0] in {$v_1$, $v_2$, $v_3$, $v_4$, $v_5$, $v_6$} {
    \pgfmathparse{90 + \idx * (360 / 6)}
    \node at (\pgfmathresult:1.8cm) {\x};
  };
  
  \node [left] at (-0.65,1.55) {-1};

\end{tikzpicture}
\end{subfigure}
\end{figure}

Now consider a sequence of voter models on the graphs $\{G_{n}\}_{n=1}^\infty$ with some sequence of initial condition vectors $\boldsymbol{g}^{(n)}\in \mathbb{R}^{n}$. Then the unique solution is determined by (\ref{altIVPsol}) and consensus is determined by the eigenvalues of the matrix $D^{(n)}$. The matrices $D^{(n)}$ for $n=5$ and $6$ are computed below.

\begin{align*}
D^{(5)}&=\frac{1}{5} \begin{bmatrix}
-2/3 & -1/3  & 1      &1&1     \\
-1/3 & -8/3& 1& 1 & 1    \\
1 &1 & -4 &1& 1 \\
1   & 1 & 1 &-4 &1  \\
1  & 1&1&1& -4
\end{bmatrix}\\
&\\
D^{(6)}&=\frac{1}{6} \begin{bmatrix}
-3 & -1  & 1      &1&1   &1  \\
-1 & -3& 1& 1 & 1 &1   \\
1 &1 & -5&1& 1 &1\\
1   & 1 & 1 &-5 &1  &1\\
1   & 1 & 1 &1 &-5 &1\\
1  & 1&1&1& 1&-5
\end{bmatrix}.
\end{align*}

 Computing the eigenvalues for the matrix $D^{(n)}$ and hence solving the consensus probelem will become infeasible as $n$ grows sufficiently large.\hfill $\diamond$

\end{Example}

This paper we will develop a method to approximately solve the consensus problem on large graphs, which applies to the graph sequence considered in the above example.

%%%%%%%%%%%%%%%%%%%%%%%%%%%%%%%%%%%%%%%%%%%%%%%%%%%%%%%%%%%%%%%%%%%%%

\section{Graph limit theory}\label{s3}

%%%%%%%%%%%%%%%%%%%%%%%%%%%%%%%%%%%%%%%%%%%%%%%%%%%%%%%%%%%%%%%%%%%%%

 The theory of graph limits was introduced by Lov\'asz and Szegedy in 2006 \cite{Lov1} and then further developed in a series of papers by Borgs et al. \cite{Borgs11, Borgs1}. A key goal of Lov\'asz and Szegedy was to understand large graph structures by characterising convergence for sequences of graphs which grow unboundedly, thereby constructing a natural `limit object'. 

In \Cref{s3.1} we present some basic definitions and results from the theory of graph limits, and describe a canonical example of a convergent graph sequence. In \Cref{s3.3} we review the appropriate notion of connectivity for graph limits. For an overview of the theory and its applications we refer the interested reader to the excellent monograph by Lov\'asz \cite{Lov}.

%%%%%%%%%%%%%%%%%%%%%%%%%%%%%%%%%%%%%%%%%%%%%%%%%%%%%%%%%%%%%%%%%%%%%

\subsection{Preliminaries}\label{s3.1}

%%%%%%%%%%%%%%%%%%%%%%%%%%%%%%%%%%%%%%%%%%%%%%%%%%%%%%%%%%%%%%%%%%%%%

The theory of graph limits is based on the graph-theoretic notion of a homomorphism density. This notion is then used to define convergence of graph sequences. Alternatively, we will focus on an equivalent definition formulated using a metric which we review in this section. 

Throughout this paper we will use the term measurable, this will always refer to the Lebesgue measure with Borel $\sigma$-algebra. It is also important to note that a number of the definitions and results reviewed in this section have been stated without an additional necessary requirement of \emph{uniformly bounded edge weights} since this is automatically satisfied by our assumption that edge weights are contained in $[-1,1]$. Lastly, since this paper is only concerned with edge-weighted graphs rather than both vertex and edge weighted graphs, we will restrict some definitions to this special case without explicit warning. In what follows the phrase `weighted graph' will always refer to an edge-weighted graph.

Let $H$ be a weighted graph, the \emph{pixel kernel of $H$} is a measurable function from $[0,1]^2$ to $[-1,1]$ which represents the graph $H$ and is denoted by $W_H$. The construction is as follows: let $H=(V,E)$ be a weighted graph with edge weights $\{\beta_{ij}\}_{i,\,j\in V}$. Partition $[0,1]$ into $|V|$ measurable subsets of equal measure, say $\{I_i\}_{i\in V}$. Then define the pixel kernel of $H$:
\begin{align}\label{pixel}
W_H(x,y)&=\beta_{ij}  \qquad \text{   if } (x,y)\in I_i\times I_j \text{ for some $i,j\in V$}.
\end{align}
This construction is not unique, however given a graph, the set of pixel kernels arising via (\ref{pixel}) forms an equivalence class under the weakly isomorphic relation (which we do not define here).

We now define kernels which are the limit objects of graph sequences.
\begin{definition}
A kernel is symmetric, measurable function $W:[0,1]^2\rightarrow[-1,1]$. In the special case that $W: [0,1]^2\rightarrow [0,1] $ we call $W$ a \emph{graphon}. We denote the set of all kernels and graphons by $\mathcal{W}_1$ and $\mathcal{W}_0$, respectively.
\end{definition}

Informally a kernel (or graphon) can be thought of as a generalisation of the adjacency matrix of a weighted graph which has a continuum number of vertices. 

\begin{definition}
A kernel which is also a step function is called a \emph{step-kernel}.
\end{definition}

Note that for every weighted graph $H$ the pixel kernel of $H$ (\ref{pixel}) is a step kernel.

In the theory of graph limits convergence of graph sequences can be defined via the \emph{cut-distance} metric, $\delta_\Box(\cdot\, , \, \cdot)$. However, for the purposes of this paper a stronger form of convergence - convergence in $L_2$-norm suffices (refer to \cite[Equation 8.14]{Lov} and \cite[Theorem 2.5]{Borgs11}).

\begin{definition}
Let $\{H_n\}_{n\in \mathbb{N}}$ be a sequence of weighted graphs and let $W\in \mathcal{W}_1$ be a kernel. We say that the graph sequence $\{H_n\}_{n\in \mathbb{N}}$ converges to the kernel $W$ if the sequence of pixel kernels $\{W_{H_n}\}_{n\in \mathbb{N}}$ converges to $W$ in $L_2$-norm.
\end{definition} 

We now present a canonical example of a graph sequence which converges to a given kernel. This construction was utilised by \cite{Med} and will be referred to throughout this paper.

\begin{Example} \cite[Section 5]{Med}
\label{approxgraphonforivp}
Given $W\in \mathcal{W}_1$, for every positive integer $n$ define the partition $\mathcal{P}_n$ of $[0,1]$ by
\begin{align*}
\mathcal{P}_n=\Bigg\{I_i^n: I_i^n=\bigg(\frac{i-1}{n}, \, \frac{i}{n}\bigg]\,\text{ for } i\in [n]\Bigg\}.
\end{align*} 
For completeness we can redefine $I_1^n=\big[0, \frac{1}{n}\big]$, however, it makes no difference for the results that follow. We now construct the sequence of weighted graphs $H_n$ on vertex set $[n]$ with edge weights 
\begin{align*}
\beta_{ij}^n&=n^2 \int_{I_i^n\times I_j^n} W(x,y)\,dx\,dy \qquad\text{for all $i,j\in [n].$}
\end{align*}
Medvedev \cite{Med} proved that $W_{H_n}$ converges to $W$ in $L_2$-norm and hence the graph sequence $\{H_n\}_{n\in \mathbb{N}}$ converges to the kernel $W$.\hfill $\diamond$
\end{Example}

%%%%%%%%%%%%%%%%%%%%%%%%%%%%%%%%%%%%

\subsection{Connectivity in graph limit}\label{s3.3}

%%%%%%%%%%%%%%%%%%%%%%%%%%%%%%%%%%%%

We will make use of the notion of connectivity for graph limits, introduced by Janson \cite{Jan1}, which we review here. In fact it will be shown in \Cref{degenexample} that connectivity of the graph limit, or kernel, is a necessary condition for consensus when considering arbitrary initial condition functions in the continuum voter model (to be defined in \Cref{s4}). 

\begin{definition} {\cite[Definition 1.12]{Jan1}}
\label{connectedgraphon}\\
A kernel $W\in \mathcal{W}_1$ is \em connected \em if for every measurable subset $S\subseteq [0,1]$ with $0<\lambda (S)<1$, we have
\begin{align*}
\int_{S\times \bigl([0,1]\backslash S\bigr)} |W(x,y)|\,dx\,dy> 0.
\end{align*}
Here $\lambda(S)$ denotes the Lebesgue measure of the set $S\subseteq [0,1]$.
\end{definition}

Janson proved that every kernel can be decomposed as a direct sum of connected kernels. To describe this decomposition we need some definitions. Given an interval $J=[M,N]\subseteq [0,1]$, we define the nonnegative linear function from $J$ to $[0,1]$ as follows:
\begin{align}
\phi: J&\rightarrow [0,1]\notag\\\label{affinemap}
x&\mapsto \frac{x-M}{M-N}=\frac{x-M}{\lambda(J)}.
\end{align}

\begin{definition}
\label{directsum}
Let $\{W_i\}_{i\in I}\subseteq \mathcal{W}_1$ be a countable family of kernels and let $\{a_i\}_{i\in I}$ be a countable family of real postive numbers such that $\sum_{j\in I} a_i=1$. Then the direct sum of $\{W_i\}_{i\in I}$ with weights $\{a_i\}_{i\in I}$ is the kernel denoted by
\begin{align}
\label{directsumeq}
\bigoplus_{i\in I} a_iW_i\in \mathcal{W}_1.
\end{align}
We interpret (\ref{directsumeq}) by partitioning $[0,1]$ into intervals $J_i$ of Lebesgue measure $a_i$ for $i\in I$ and denoting the nonnegative linear function from $J_i$ to $[0,1]$ by $\phi_i$ (see (\ref{affinemap})). Then 
\begin{align}
\label{directsumexplicit}
\Big(\bigoplus_{i\in I} a_iW_i \Big)(x,y)=
\begin{cases}
W_k\bigr(\phi_k(x),\phi_k(y)\bigl) &\text{ if } x,y\in J_k \text{ for some $k\in \mathbb{Z}^+$,}\\
0 & \text{ otherwise}.
\end{cases}
\end{align}
\end{definition}

\begin{Lemma} \emph{\cite[Theorem 1.5]{Jan1}}
\label{decompker}\\
Let $W\in \mathcal{W}_1$. Then there exist a countable family of connected kernels, \linebreak$\{W_i\}_{i\in I}\subseteq \mathcal{W}_1$, and a corresponding family of positive real numbers $\{a_i\}_{i\in I}$ with $\sum_{i\in I} a_i=1$ such that 
\begin{align*}
W=\bigoplus_{i\in I} a_iW_i.
\end{align*}
The connected kernels $\{W_i\}_{i\in I}$ will be referred to as the connected components of $W$.
\end{Lemma}

%%%%%%%%%%%%%%%%%%%%%%%%%%%%%%%%%%%%%%%%%%%%%%%%%%%%%%%%%%%%%%%%%%%%%

\section{Voting on large graphs and graph limits}\label{s4}

%%%%%%%%%%%%%%%%%%%%%%%%%%%%%%%%%%%%%%%%%%%%%%%%%%%%%%%%%%%%%%%%%%%%%

Extending the voting model to kernels is motivated by a number of observations. Firstly, it is a more general setting to study the voting model which includes the finite voting model considered in \Cref{s2} as a special case. Secondly, many modern-day networks such as the internet are for all practical purposes infinitely large. Thus it may be more appropriate to consider the voting model in the setting of graph limits. Thirdly, the behaviour of solutions to the finite voting model are completely determined via eigenvalues and eigenvectors (recall (\ref{altIVPsol})). However, for sufficiently large graphs it is computationally infeasible to compute these values and vectors. Approximating a large graph by a kernel provides an alternative method for approximately solving the dynamics of the IVP (\ref{IVP}) on large graphs.

 We begin by presenting the voting model on a kernel. Let $W\in \mathcal{W}_1$ and \linebreak $\boldsymbol{g}\in L_\infty([0,1])$. Then the continuum limit of (\ref{IVP}) can be expressed by\linebreak $\boldsymbol{u}: \mathbb{R}^{\ge 0}\rightarrow L_\infty([0,1])$ such that
\begin{align}
\label{PDE}
\begin{cases}
\frac{\partial u(x,t)}{\partial t}& = \int_0^1 W(x,y)\Big(u(y,t)-u(x,t)\Big)dy \qquad \text{for all } x\in [0,1] \text{ and }  t\in \mathbb{R}^{>0},\\
\boldsymbol{u}(0)&=\boldsymbol{g}.
\end{cases}
\end{align}

Here, and throughout this paper, the integral sign refers to Lebesgue integration. Replacing IVPs such as (\ref{IVP}) with the continuum limit has attracted interest in a number of papers \cite{Fren, Ome, Stro}. However as pointed out by Medvedev \cite{Med}, ``a rigorous justification for taking the continuum limit in (such models) was lacking". In light of this, Medvedev \cite{Med} showed that the theory of graph limits could be used to prove that the continuum limit of such dynamical systems can approximate the dynamics on large finite graphs, under certain conditions.

In \Cref{s6} we will show how the continuum IVP (\ref{PDE}) can be used to approximate the consensus problem of the finite voter model on large graphs. In particular, we will revisit \Cref{Example1} to illustrate our results.

First we present an existence and uniqueness result which shows that the continuum limit IVP (\ref{PDE}) is well-posed; allowing the voting model to be extended to the continuum setting. The result considers the space of continuously differentiable vector-valued functions from $\mathbb{R}$ to $L_\infty([0,1])$, denoted by $C^1(\mathbb{R}, L_\infty([0,1]))$. This space is equipped with the following norm: let $\boldsymbol{u}\in C^1(\mathbb{R}, L_\infty([0,1])) $ then
\begin{align}\label{continuousnorm}
\|\boldsymbol{u}\|_{C(\mathbb{R}, L_\infty([0,1]))}&=\sup_{0\le t<\infty} \|\boldsymbol{u}(t)\|_\infty= \sup_{0\le t<\infty} \underset{x\in [0,1]}{\mathrm{ess\ sup} }\,|u(x,t)|,
\end{align}
where $\mathrm{ess\ sup}$ denotes the essential supremum. The theorem follows as a special case of \cite[Theorem 3.2]{Med}. 

\begin{Theorem}\label{existanduniq} \emph{(Existence and Uniqueness)}\\
Let $W\in \mathcal{W}_1$ and $\boldsymbol{g}\in L_\infty\big([0,1]\big)$. Then the IVP (\ref{PDE}) has a unique solution in $C^1\Big(\mathbb{R}, L_\infty\big([0,1]\big)\Big)$.
\end{Theorem}

For the remainder of this paper when we refer to `the' solution of the IVP (\ref{PDE}) this will always mean the unique solution in $C^1(\mathbb{R}, L_\infty([0,1]))$ described above. We now introduce a continuum analog of \Cref{consensus}.

\begin{definition}\label{condef} Let $A\subseteq [0,1]$ of nonzero measure. We say that \em consensus on $A$ \em is attained if 
\begin{align}
\label{consensuseq1}
\lim_{t\rightarrow \infty}\, \underset{(x,y)\in A^2}{\mathrm{ess\ sup} } \,|u(x,t)-u(y,t)|=0.
\end{align}
If $A=[0,1]$ then we say \em consensus \em is attained.
\end{definition}

\begin{Example}
The constant $1$-valued graphon attains consensus for every \linebreak$\boldsymbol{g}\in L_\infty([0,1])$, while, the constant $-1$-valued kernel reaches consensus if and only if $\boldsymbol{g}$ is a constant valued function.\hfill $\diamond$
\end{Example}

By applying Medvedev's convergence result \cite[Theorem 5.2]{Med} we justify the use of the continuum limit (\ref{PDE}). We show that under certain conditions the continuum limit approximates the solutions to the voting model on sufficiently large graphs. First consider an IVP in the form of (\ref{PDE}) with kernel $W$ and initial condition function $\boldsymbol{g}\in L_\infty([0,1])$. We define a sequence of IVPs with the $n$-th IVP corresponding to the voting model on the graph $H_n$ (as defined in \Cref{approxgraphonforivp}) and initial condition $\boldsymbol{g}_n$ defined below:
\begin{align}
\label{gsubn}
g_n(x)=\int_{\frac{i-1}{n}}^{\frac{i}{n}} g(y)\,dy \qquad\text{for each $i\in[n]$ and for all $x\in I_i^n$},
\end{align}
where $I_i^n=\Big(\frac{i-1}{n},\frac{i}{n}\Big]$. This can naturally be interpreted as a vector in $\mathbb{R}^n$, say $\boldsymbol{g}^{(n)}$ for each $n$. Where the $i$-th component of the vector $\boldsymbol{g}^n$ is the value of the function $\boldsymbol{g}_n$ on $I_i^n$.

Thus, the $n$-th approximate IVP is
\begin{align}
\label{IVPapprox}
\begin{cases}
\frac{du_i^n(t)}{dt}&=\frac{1}{n}\sum_{j=1}^n  \beta^n_{ij}\big(u_j^n(t)-u_i^n(t)\big)\qquad \text{for all }\, i\in [n] \text{ and } t\in \mathbb{R}^{>0},\\
\boldsymbol{u}^n(0)&=\boldsymbol{g}^{(n)}.
\end{cases}
\end{align}

 This is equivalent to the continuum IVP with step kernel $W_n=W_{H_n}$ (see (\ref{pixel})) and step function $\boldsymbol{g}_n$ (described in (\ref{gsubn}))
\begin{align}
\label{approxPDE}
\begin{cases}
\frac{\partial u_n(x,t)}{\partial t}&=\int_0^1 W_n(x,y)\big(u_n(y,t)-u_n(x,t)\big)\,dy\qquad \text{for all } x\in [0,1] \text{ and }  t\in \mathbb{R}^{>0},\\
\boldsymbol{u}_n(x,0)&=\boldsymbol{g}_n(x),
\end{cases}
\end{align}
in the sense that $u_n(x,t)=u_i^n(t)$ for all $x\in I_i^n$.

In a similar manner to (\ref{continuousnorm}), we define the $C([0,T]; L_2([0,1]))$-norm as follows: let $\boldsymbol{u}\in C([0,T]; L_2([0,1]))$ then
\begin{align*}
\|\boldsymbol{u}\|_{C([0,T];L_2([0,1]))}=\max_{t\in [0,T]}\|\boldsymbol{u}(t)\|_{2}.
\end{align*} 
We now present a key result, which follows as an application of \cite[Theorem 5.2]{Med}.

\begin{Theorem} \label{approx} 
Let $W\in \mathcal{W}_1$ and let $\boldsymbol{g}\in L_\infty([0,1])$. If $\{\boldsymbol{u}_n\}_{n\in \mathbb{N}}$ is a sequence of solutions to IVPs of the form (\ref{PDE}) with initial conditions functions \linebreak$\{\boldsymbol{g_n}\}_{n\in \mathbb{N}}\subseteq L_\infty([0,1])$ and kernels $\{W_n\}_{n\in \mathbb{N}}\subseteq \mathcal{W}_1$ and $\boldsymbol{u}$ is a solution to the IVP (\ref{PDE}) with initial condition $\boldsymbol{g}$ and kernel $W$, then for any $0<T<\infty$
\begin{align}\label{med1}
\|\boldsymbol{u}-\boldsymbol{u}_n\|_{C([0,T];L_2([0,1]))}&\le\Bigg(\|\boldsymbol{g}-\boldsymbol{g}_n\|_2^2 +\frac{C_1\|W-W_n\|_2}{C_2}\Bigg)\mathrm{exp}(C_2 \,T),
\end{align}
where $C_1,\, C_2$ are positive constants independent of $n$. Furthermore, if each $W_n$ and $\boldsymbol{g}_n$ is given by the construction in \Cref{approxgraphonforivp} and (\ref{gsubn}) then for any $0<T<\infty$
\begin{align*}
\|\boldsymbol{u}-\boldsymbol{u}_n\|_{C([0,T];L_2([0,1]))}&\rightarrow 0\qquad\text{as $n\rightarrow \infty$.}
\end{align*}
\end{Theorem}

For the voting model, \Cref{approx} shows that solutions to the IVP (\ref{IVPapprox}) can be approximated by the graph limit IVP (\ref{PDE}). However, in general replacing IVPs on sequences of convergent graphs with the IVP on their graph limit does not guarantee convergence \cite{medtwisted}.

%%%%%%%%%%%%%%%%%%%%%%%%%%%%%%%%%%%%%%%%%%%%%%%%%%%%

\section{Reaching consensus}\label{s6}

%%%%%%%%%%%%%%%%%%%%%%%%%%%%%%%%%%%%%%%%%%%%%%%%%%%%

In \Cref{s4} using the results of \cite{Med} we were able to show that the continuum voting model could be used to approximate the finite voting model for large graphs. However, this approximation is attained in the $C([0,T], L_2([0,1]))$-norm, and consensus of the continuum model does not imply consensus of the finite model, even for large graphs. This leaves open the question of whether graph limit theory can be used to infer consensus in the finite voting model.

We now prove our main result which answers the above question in the affirmative. We show that if the solution to the continuum model (\ref{PDE}) reaches consensus then solutions to the finite model (\ref{approxPDE}) will be close to a constant function, for sufficiently large $n$ and sufficiently large $t$.  This provides motivation for the continued study of the continuum model and, in particular, the search for sufficient conditions which guarantee consensus - this is pursued in \Cref{s6.1}. 

First we present a lemma which will be used the proof of our main result \Cref{keytheorem}. 

\begin{Lemma}\label{conservation} Let $W\in\mathcal{W}_1$ be a kernel and let $\boldsymbol{g}\in L_\infty([0,1])$. If $\boldsymbol{u}$ is the solution of the IVP (\ref{PDE}) then 
\begin{align}
\label{conservationeqlemma}
\int_0^1 u(x,t)\,dx&=\int_0^1 g(x)\,dx \qquad\text{for any $0\le t<\infty$.}
\end{align}
\end{Lemma}

\begin{proof} 
This proof follows similarly to that of \cite[Lemma 3.5]{Med1}.
\end{proof}

\begin{Theorem} \label{keytheorem}
Let $\boldsymbol{u}$ be a solution to the IVP (\ref{PDE}), with kernel $W\in\mathcal{W}_1$ and initial condition function $\boldsymbol{g}\in L_\infty([0,1])$, and let $\boldsymbol{u}_n$ be the solution to the approximate IVP (\ref{approxPDE}) for each positive integer $n$. Suppose that $\boldsymbol{u}$ reaches consensus on $[0,1]$ and let $D$ be any positive real number. Then for every $\varepsilon>0$ and for every $c>0$, there exists $T=T(\varepsilon)$ and a subset $S_t\subseteq[0,1]^2$ with $\lambda(S_t)<c^2$ such that for all sufficiently large $n$,
\begin{align}
\label{res}
|u_n(x,t)-u_n(y,t)|\le \varepsilon \qquad \text{ for all $(x,y)\in [0,1]^2\setminus S_t$ and $t\in [T,T+D]$.}
\end{align}
\end{Theorem}

\begin{proof}
For any positive integer $n$, applying the triangle inequality three times gives the following upper bound on the consensus equation namely, (\ref{consensuseq1}) with $A=[0,1]$, of $\boldsymbol{u}_n$, for any fixed $t\in \mathbb{R}^{\ge 0}$:
\begin{align}
\label{attack}
&\underset{(x,y)\in [0,1]^2}{\mathrm{ess\ sup} }\,|u_n(x,t)-u_n(y,t)|\notag\\
&\quad\le 2\,\underset{x\in [0,1]}{\mathrm{ess\ sup} }\,|u_n(x,t)-u(x,t)|+\underset{(x,y)\in [0,1]^2}{\mathrm{ess\ sup} }\,|u(x,t)-u(y,t)|.
\end{align}

Since $\boldsymbol{u}$ reaches consensus, for any $\varepsilon>0$ there exists $T>0$ such that 
\begin{align}
\label{bound2}
\underset{(x,y)\in [0,1]^2}{\mathrm{ess\ sup} }\, |u(x,t)-u(y,t)|<\frac{\varepsilon}{3} \qquad\text{for all }\, t\ge T.
\end{align} 
Let $X$ be a random variable uniformly distributed in $[0,1]$, denoted as $X\sim U[0,1]$, and define the continuous-time processes $\{u(X,t)\}_{t\ge 0}$ and $\{u_n(X,t)\}_{t\ge 0}$. By \linebreak \Cref{conservation}, we have $\mathbb{E}[u(X,t)-u_n(X,t)]=\int \boldsymbol{g}-\int\boldsymbol{g}_n=0$. Then Chebyshev's inequality \cite[Section 7.3]{Will} gives, for all $t>0$,
\begin{align}
\label{bound0}
\mathbb{P}\Big[\big|u(X,t)-u_n(X,t)\big|>\varepsilon\Big] &\le \frac{\mathbb{E}\big[\big(u(X,t)-u_n(X,t)\big)^2\big]}{\varepsilon^2}\notag\\
&=\frac{\|\boldsymbol{u}(t)-\boldsymbol{u}_n(t)\|^2_{2}}{\varepsilon^2}.
\end{align}

 Now for any $D>0$, we have $\boldsymbol{u}_n\rightarrow \boldsymbol{u}$ in $C([0, T+D]; L_2([0,1]))$-norm (\Cref{approx}), and from (\ref{bound0}) we have
\begin{align}
\label{bound3}
\sup_{t\in [0,T+D]} \mathbb{P}\Big[\big|u(X,t)-u_n(X,t)\big|>\nfrac{\varepsilon}{3}\Big]&\le\Big(\frac{3\|\boldsymbol{u}-\boldsymbol{u}_n\|_{C([0, T+D]; L_2([0,1])}}{\varepsilon}\Big)^2\rightarrow 0,
\end{align}
for large $n$. Thus, for any $\varepsilon>0$ and any $c>0$ there exists an $N$ such that for all $n>N$,
\begin{align}
\label{bound4}
\|\boldsymbol{u}-\boldsymbol{u}_n\|_{C([0, T+D]; L_2([0,1]))}<\frac{c\,\varepsilon}{3\sqrt{2}}.
\end{align}
Combining (\ref{bound0}) and (\ref{bound4}), we see that for each $t\in [0, T+D]$, the set
$$A_t=\Big\{x\in [0,1]: \big|u(x,t)-u_n(x,t)\big|>\nfrac{\varepsilon}{3}\Big\}$$
has Lebesgue measure at most $\nicefrac{c^2}{2}.$ It follows that 
$$S_t=\big\{(x,y)\in [0,1]^2: x\in A_t \text{ or } y\in A_t\big\}$$
has Lebesgue measure strictly less than $c^2$.

Applying (\ref{attack}) and (\ref{bound2}) we attain the required result: for every $\varepsilon>0$ and for every $c>0$ there exists a positive integer $N$ such that for every $t\in [T,T+D]$, if $n>N$ then 
\begin{align*}
|u_n(x,t)-u_n(y,t)|\le \varepsilon \qquad \text{for all }\, (x,y)\in [0,1]^2\backslash S_t. \tag*{\qedhere}
\end{align*}
\end{proof}

To illustrate the above result we return to a generalised version of the consensus problem introduced in \Cref{Example1}.

\begin{Example}
For each $0<r<\frac{1}{2}$ define the kernel 
\begin{align}\label{bipartitekernel}
W_r(x,y)&=\begin{cases}
-1 &\text{if $x, y<r$}\\
1 &\text{otherwise,}
\end{cases}
\end{align}
and let $\mathcal{G}_r\in L_\infty([0,1])$ be the family of functions such that 
\begin{align*}
\int_0^r g_r(y)=\int_r^1 g_r(y)=0.
\end{align*}
Note that when $r=\frac{1}{3}$ we attain the function $W$ introduced in \Cref{Example1} and illustrated in \Cref{fig1}.

For any positive integer $n$ and when $r=\frac{1}{3}$ the consensus problem from \Cref{Example1} can be attained via the approximating process defined in (\ref{IVPapprox}) and (\ref{approxPDE}). Thus, for any initial condition function $\boldsymbol{g}\in L_\infty([0,1])$ we can apply \Cref{keytheorem} to approximately solve the consensus problem by considering the continuum IVP (\ref{PDE}) with kernel (\ref{bipartitekernel}).

The unique solution of the continuum IVP (\ref{PDE}) with kernel $W_r$  and initial condition function $\boldsymbol{g}\in \mathcal{G}_r$ is
\begin{align*}
u(x,t)&=\begin{cases}
g(x) \, e^{-t} &\text{if $x\ge r$}\\
g(x) \, e^{-(1-2r)t} &\text{if $x< r$.}
\end{cases}
\end{align*}
This solution can be easily verified to satisfy (\ref{PDE}). The piecewise solution reflects the structure of the kernel $W_r$ defined in (\ref{bipartitekernel}) with $x\in [0,r)$ decaying at a slower rate than $x\in [r,1]$. This is due to the negative value of $W_r(x,y)$ for $x,y<r$ which, informally speaking, produces a push away from consensus. However, as can be observed when $0<r<\frac{1}{2}$, the solution $u(x,t)$ converges to zero for all $x\in [0,1]$ as $t\rightarrow \infty$ and hence consensus is attained.

Thus, we conclude that consensus is attained in the continuum model and so by \Cref{keytheorem} we can get arbitrarily close to consensus in the discrete model for sufficiently large graphs. \hfill $\diamond$
\end{Example}

 The above example illustrates an application of \Cref{keytheorem}. However, finding kernels which attain consensus such as (\ref{bipartitekernel}) so that \Cref{keytheorem} can be applied is a non-trivial task.  The remainder of this section and \Cref{s6.1} focuses on understanding solutions to the continuum voter model and characterising necessary and sufficient conditions for consensus.

 We now establish an almost everywhere pointwise limit for solutions to the continuum model. Recall that vector-valued functions are written in bold font.

\begin{definition}\label{boundedsolution}
Let $\boldsymbol{u}$ be a solution to the IVP (\ref{PDE}). We say that $\boldsymbol{u}$ is a \em bounded solution \em if 
\begin{align*}
\sup_{t\ge 0} \,\|\boldsymbol{u}(t)\|_{\infty}=\sup_{t\ge 0}\Big\{ \underset{x\in [0,1]}{\mathrm{ess\ sup} }\,|u(x,t)|\Big\}<\infty.
\end{align*}
\end{definition}

In the continuum model, if $W\in \mathcal{W}_0$ is a graphon then the solution $\boldsymbol{u}$ is automatically bounded.

\begin{Lemma}\label{graphonboundedsol} \cite[Theorem 3.1]{Ignat}\\
Let $ W\in \mathcal{W}_0$ and let $\boldsymbol{g}\in L_\infty([0,1])$. If $\boldsymbol{u}$ solves the IVP (\ref{PDE}) with $W$ and $\boldsymbol{g}$ then 
\begin{align*}
\|\boldsymbol{u}(t)\|_{\infty}\le \|\boldsymbol{g}\|_{\infty} \qquad\text{for all }\, t\ge 0.
\end{align*} 
Hence $\boldsymbol{u}$ is a bounded solution.
\end{Lemma}

We now show that any bounded solution to the continuum model (\ref{PDE}) is a continuous-time martingale. It is important to note that \Cref{graphonboundedsol} only applies to the IVP (\ref{PDE}) when $W\in \mathcal{W}_0 \subsetneq \mathcal{W}_1$ is a graphon - a nonnegative valued kernel. Whilst, \Cref{ctsmart} (to be presented) applies to the IVP (\ref{PDE}) for any $W\in \mathcal{W}_1$ and $\boldsymbol{g}\in L_\infty([0,1])$ which admits a bounded solution.

The following definition involves the conditional expectation operator which is covered in detail by \cite[Chapter 9]{Will}. For the purposes of this paper we need only briefly explain the notation; let $X$ be a random variable, we denote the conditional expectation of $X$ given another random variable $Y$ by the random variable
\begin{align}\label{condexp}
Z=\mathbb{E}\big[ X \big| Y\big].
\end{align}
If $Y(\omega)=y_j$ for some $\omega$ in the probability space of $Y$ then we interpret (\ref{condexp}) as
\begin{align*}
Z(\omega)&=\mathbb{E}\big[ X \big| Y=y_j\big].
\end{align*}
See for full details \cite[Chapter 9]{Will}.

\begin{definition}
A family of random variables $\{Z_t:\, t\ge 0\}$ is called a continuous-time \em martingale \em  if 
\begin{align}
 \mathbb{E}\big[|Z_t|\big]&<\infty \qquad \text{ for all $t\ge 0$, and}\\
\label{martcond}
\mathbb{E}\big[Z_t\big| \{Z_r: \, 0\le r\le s\}\big]&=Z_s  \qquad\text{ for all }\,  0\le s<t .
\end{align}
\end{definition}

Recall $X \sim U[0,\, 1]$ denotes a random variable $X$ uniformly distributed in the interval $[0,\, 1]$.

\begin{Proposition}\label{ctsmart} Let $\boldsymbol{u}$ be a solution to the IVP (\ref{PDE}) with kernel $W\in\mathcal{W}_1$ and initial condition function $\boldsymbol{g}\in L_\infty([0,1])$. If $\boldsymbol{u}$ is a bounded solution and $X\sim U[0,\,1]$ then the continuous-time process $\{Z_t\}_{t\ge 0}$ defined by $Z_t=u(X,t)$ is a bounded continuous-time martingale. Further, $\{Z_t\}_{t\ge 0}$ converges almost surely as $t\rightarrow \infty$.
\end{Proposition}

\begin{proof}
First note that by assumption the solution $\boldsymbol{u}$ is bounded and so $Z_t$ is bounded. That is, 
\begin{align*}
\sup_{t\ge0}\,\mathbb{E}[|Z_t|]&=\sup_{t\ge0}\,\int_0^1 |u(x,t)|\,dx\le \sup_{t\ge 0} \, \|\boldsymbol{u}(t)\|_\infty<\infty.
\end{align*}
Secondly, we have for all $x\in [0,1]$ and for all $t\in \mathbb{R}^{\ge 0}$,
\begin{align*}
u(x,t)&=u(x,s)+\int_s^t\int_0^1 W(x,y)\big(u(y,r)-u(x,r)\big)\,dy\,dr\ \text{ for any $0\le s\le t$}.
\end{align*}
Thus, taking the conditional expectation gives
\begin{align*}
\mathbb{E}[Z_t| \{Z_\tau, \tau\le s\}]&=Z_s+\mathbb{E}\Big[\int_s^t\int_0^1 W(X,y)\big(u(y,r)-u(X,r)\big)\,dy\,dr\,\Big|\{Z_\tau, \tau\le s\}\Big]\\
&=Z_s+\int_0^1\int_s^t\int_0^1 W(x,y)\big(u(y,r)-u(x,r)\big)\,dy\,dr\,dx\\
&=Z_s+\int_0^1\int_s^t\int_0^1 W(x,y)\,u(y,r)\,dy\,dr\,dx\\
&\qquad\,\,-\int_0^1\int_s^t\int_0^1W(x,y)\,u(x,r)\,dy\,dr\,dx.
\end{align*}
Since the integrands of the last two terms are absolutely integrable and bounded above by some positive constant, via Fubini's theorem \cite[Thereom 2.37]{Foll} we have that
\begin{align*}
&\int_0^1\int_s^t\int_0^1 W(x,y)\,u(y,r)\,dy\,dr\,dx-\int_0^1\int_s^t\int_0^1W(x,y)\,u(x,r)\,dy\,dr\,dx\\
&=\int_s^t\Bigg(\int_{[0,1]^2} W(x,y)\,u(y,r)\,dy\,dx-\int_{[0,1]^2}W(x,y)\,u(x,r)\,dx\,dy\Bigg)\,dr.
\end{align*}
Noting that $W$ is symmetric we have that the above expression equals zero. Hence the necessary condition (\ref{martcond}) holds and we have a continuous-time martingale.

The final claim in the corollary follows immediately from the Martingale Convergence Theorem \cite[Theorem 11.5]{Will}.
\end{proof}

\begin{Remark}\label{remarkctsmart}
 The final statement in the proposition above can equivalently be stated in deterministic terms: for all bounded solutions $\boldsymbol{u}$ to the IVP (\ref{PDE}) the limit $\lim_{t\rightarrow \infty} u(x,t)$ exists for almost every $x\in [0,1]$, and the function $\lim_{t\rightarrow \infty}\boldsymbol{u}(t)$ is integrable.
 \end{Remark}

 We now define the following function:
\begin{align*}
\boldsymbol{u}^*: [0,1]&\rightarrow \mathbb{R}\\
x &\mapsto \begin{cases}
\lim_{t\rightarrow \infty}u(x,t)&\text{if the limit exists,}\\
0&\text{otherwise.}
\end{cases}
\end{align*}

 It should be noted that the existence of $\lim_{t\rightarrow \infty} u(x,t)$ almost everywhere is not enough to ensure that consensus is attained. A trivial example is the constant $0$-valued graphon, which will have $\boldsymbol{u}(t)=\boldsymbol{u}^*=\boldsymbol{g}$ for any initial condition function $\boldsymbol{g}\in L_\infty([0,1])$. This degeneracy is due to the fact that the constant $0$-valued graphon is not connected (as per \Cref{connectedgraphon}).  \Cref{degenexample} will show that for a general initial condition function $\boldsymbol{g}$, a necessary condition for consensus is that the kernel be connected. 
 
First we characterise the limiting value, $\boldsymbol{u}^*$, if consensus is reached. 

\begin{Proposition}\label{consisconstant}
Given a kernel $W\in\mathcal{W}_1$ and initial condition $\boldsymbol{g}\in L_\infty([0,1])$, let $\boldsymbol{u}$ be a solution to the IVP (\ref{PDE}) on $W$ and $\boldsymbol{g}$. If $\boldsymbol{u}$ reaches consensus then 
\begin{align*}
\boldsymbol{u}^*(x)&=\int_0^1\boldsymbol{g}(y)\,dy\qquad \text{for almost every $x\in [0,1]$. }
\end{align*}
\end{Proposition}

\begin{proof}
By \Cref{conservation} we have
\begin{align}
\label{equalinitial}
\int_0^1 u(x,t)dx=\int_0^1 g(y)\,dy \qquad\text{for all }0\le t<\infty.
\end{align}
If $\boldsymbol{u}$ reaches consensus then $\boldsymbol{u}$ is a bounded solution (as per \Cref{boundedsolution}) and so $\lim_{t\rightarrow \infty} u(x,t)$ exists for almost every $x\in [0,1]$ (see \Cref{ctsmart} and \Cref{remarkctsmart}). To see that $\boldsymbol{u}$ is bounded observe the following; for fixed $x\in [0,1]$
\begin{align*}
|u(x,t)|&=\Big|\int_0^1 u(x,t)-u(y,t) + u(y,t)\, dy\Big|\\
&=\Big|\int_0^1 u(x,t)-u(y,t)\, dy +\int_0^1 g(y)\, dy\Big| &&\text{from (\ref{equalinitial}),}\\
&\le \int_0^1  \big|u(x,t)-u(y,t)\big|\, dy +\Big|\int_0^1 g(y)\, dy \Big|.
\end{align*}
Now by assumption $\boldsymbol{u}$ reaches consensus and so for almost every $x\in [0,1]$
\begin{align*}
 \lim_{t\rightarrow \infty} |u(x,t)|&\le \Big|\int_0^1 g(y)\, dy \Big| \le 1,
\end{align*}
this suffices to show that $\boldsymbol{u}$ is a bounded solution.

Now to prove the statement in the proposition, we take the limit as $t$ approaches infinity of (\ref{equalinitial}) and applying the Dominated Convergence Theorem \cite[Theorem 2.24]{Foll} gives
\begin{align*}
\int_0^1 u^*(x)dx=\lim_{t\rightarrow \infty} \int_0^1 u(x,t)dx=\int_0^1 g(x)\,dx.
\end{align*}
But $\boldsymbol{u}^*$ is constant almost everywhere, since $\boldsymbol{u}$ reaches consensus. Thus 
\begin{align*}
\boldsymbol{u}^*(x)&=\int_0^1\boldsymbol{g}(y)\,dy\qquad \text{for almost every $x\in [0,1]$. }\tag*{\qedhere}
\end{align*}
\end{proof}

As mentioned above, a necessary condition for consensus is that the kernel be connected (cf. \Cref{connectedgraphon}). The proof of this follows from \Cref{consisconstant} and the fact that the solution to the continuum voter model on a kernel can be decomposed into solutions on its connected components. The lemma below proves this claim, however, first we introduce some additional notation. 

Given a map $\varphi: [0,1] \rightarrow [0,1]$, we define the pull-back of the functions \linebreak$u: [0,1]\times \mathbb{R}\rightarrow \mathbb{R}$ and $g: [0,1]\rightarrow \mathbb{R}$ as
\begin{align*}
u^\varphi(x,t)&=u(\varphi(x),t) \qquad\text{and}\qquad g^\varphi(x)=g(\varphi(x)).
\end{align*}

\begin{Lemma} \label{decomposedynamics}
Let $W\in \mathcal{W}_1$ and $\boldsymbol{g}\in L_\infty([0,1])$. If $\boldsymbol{u}$ solves the IVP (\ref{PDE}) then there exists a countable family of positive real numbers $\{a_i\}_{i\in I}$ and a unique family of functions in $\{\boldsymbol{u}_i\}_{i\in I}\subseteq C^1(\mathbb{R}, L_\infty([0,1]))$ such that
\begin{align*}
\boldsymbol{u}(t)=\bigoplus_{i\in I} a_i \boldsymbol{u}_i(t),
\end{align*}
where $\boldsymbol{u}_i$ are solutions to the IVP (\ref{PDE}) on a connected kernel. We interpret this direct sum in a similar way to \Cref{directsum}: partition $[0,1]$ into intervals $J_i$ of length $a_i$ for $i\in I$. Then letting $\phi_i$ denote the nonnegative linear function from $J_i$ to $[0,1]$ (see (\ref{affinemap})) we have
\begin{align*}
u(x,t)=\sum_{i\in I} \mathbbm{1}_{J_i}(x)\, u_i(\phi_i(x),t).
\end{align*}
\end{Lemma}

\begin{proof}
First, decompose the kernel $W\in \mathcal{W}_1$ into a direct sum of connected kernels as in \Cref{decompker}; that is,
\begin{align*}
W=\bigoplus_{i\in I} a_iW_i
\end{align*}
for a family of connected kernels $\{W_i\}_{i\in I}\subseteq \mathcal{W}_1$ and corresponding family of positive real numbers $\{a_i\}_{i\in I}$ with $\sum_{i\in I} a_i=1$. Let $\{J_i\}_{i\in I}$ denote the partition of $[0,1]$ into intervals of length $a_i$ and let $\phi_i$ denote the nonnegative linear function from $J_i$ to $[0,1]$ (see \ref{affinemap}).

 Define the \emph{normalised degree function} of $W\in\mathcal{W}_1$ as
\begin{align}\label{normdeg}
d_W(x)=\int_0^1 W(x,y)\,dy.
\end{align}
Then via an integrating factor it can be shown that the solution to the continuum model (\ref{PDE}) solves 
\begin{align}
\label{intfact}
u(x,t)=e^{-d_W(x)t}g(x)+\int_0^t\int_0^1 e^{d_W(x)(s-t)}W(x,y)u(y,s)\,dy\,ds.
\end{align}
For an arbitrary element $x\in J_i$ (see \Cref{directsum}), $d_W(x)$ can be simplified as follows:
\begin{align*}
d_W(x)&=\int_0^1 W(x,y)\,dy\\
&=\int_{J_i} W_i(\phi_i(x), \,\phi_i(y))\,dy &&\text{see (\ref{directsumexplicit}),}\\
&=\dfrac{1}{a_i}\int_0^1 W_i(\phi_i(x), z)\,dz &&\text{substitute $y$ with $\phi_i^{-1}(z)$,}\\
&=d_{\widehat{W}_i}(\phi_i(x)) &&\text{where we define $\widehat{W}_i=\frac{1}{a_i}W$.}
\end{align*}
For $x\in J_i$, the solution (\ref{intfact}) then becomes
\begin{align*}
&u(x,t)\\
&=e^{-d_{\widehat{W}_i}(\phi_i(x))t}\,g(x)+\int_0^t\int_{J_i}e^{d_{\widehat{W}_i}(\phi_i(x))(s-t)}\,W_i(\phi_i(x),\phi_i(y))\,u(y,s)\,dy\,ds\\
&=e^{-d_{\widehat{W}_i}(\phi_i(x))t}\,g(x)+\int_0^t\int_0^1 e^{d_{\widehat{W}_i}(\phi_i(x))(s-t)}\,\widehat{W}_i(\phi_i(x),z)\,u(\phi_i^{-1}(z),s)\,dz\,ds\\
&=e^{-d_{\widehat{W}_i}(\phi_i(x))t}\,g^{\phi_i^{-1}}(\phi_i(x))+\int_0^t\int_0^1 e^{d_{\widehat{W}_i}(\phi_i(x))(s-t)}\,\widehat{W}_i(\phi_i(x),z)\,u^{\phi_i^{-1}}(z,s)\,dz\,ds,
\end{align*}
where, again, we substitute $y$ with $\phi_i^{-1}(z).$ Thus for all $x\in J_i$,
\begin{align}
\label{decompproof}
u^{\phi_i^{-1}}(x,t)&=e^{-d_{\widehat{W}_i}(x)t}\,g^{\phi_i^{-1}}(x)+\int_0^t\int_0^1 e^{d_{\widehat{W}_i}(x)(s-t)}\,\widehat{W}_i(x,z)\,u^{\phi_i^{-1}}(z,s)\,dz\,ds.
\end{align}
But the function which solves (\ref{decompproof}) is the unique solution to the IVP (\ref{PDE}) in  $C^1(\mathbb{R}, L_\infty([0,1]))$ with connected kernel $\widehat{W}_i$ and initial condition function $\boldsymbol{g}^{\phi_i^{-1}}$. Thus, by defining $\boldsymbol{u}_i$ as the solution to the IVP (\ref{PDE}) with connected kernel $\widehat{W}_i$ and initial condition function $\boldsymbol{g}^{\phi_i^{-1}}$, we have
\begin{align*}
u^{\phi_i^{-1}}(x,t)&=u_i(x,t) &&\text{for all }\, x\in J_i,\\
\text{and so }\quad u(x,t)&=u_i(\phi_i(x),t) &&\text{for all }\, x\in J_i.
\end{align*}
 Compactly, we can express this result as 
\begin{align*}
\boldsymbol{u}(t)=\bigoplus_{i\in I} a_i \boldsymbol{u}_i(t), \qquad\text{as required.} \tag*{\qedhere}
\end{align*}
\end{proof}

We now present a necessary condition for consensus in the continuum model.

 \begin{Theorem} \label{degenexample} 
Let $W$ be a kernel with connected components $\{W_i\}_{i\in I}$ and associated sets $\{J_i\}_{i\in I}$ and let $\boldsymbol{g}\in L_\infty([0,1])$. A necessary condition for the IVP (\ref{PDE}) to reach consensus is that 
\begin{align}\label{neccon}
\lambda(J_i)\,\int_{J_i} g(y) \, dy&=C &&\text{for all $i\in I$,} 
\end{align}
where $\lambda(J_i)$ denotes the Lebesgue measure of the set $J_i$. This necessary condition is immediately satisfied if $W$ is connected.
 \end{Theorem}
 
 \begin{proof}
Let $\boldsymbol{u}(t)$ be the solution to the IVP (\ref{PDE}) with kernel $W$ and initial condition function $\boldsymbol{g}$. Furthermore, suppose that $\boldsymbol{u}(t)$ reaches consensus. It follows from \Cref{decomposedynamics} that
 \begin{align*}
u(x,t)=\sum_{i\in I} \mathbbm{1}_{J_i}(x)\, u_i(\phi_i(x),t),
\end{align*}
and so
 \begin{align*}
u^*(x)=\sum_{i\in I} \mathbbm{1}_{J_i}(x)\, u_i^*(\phi_i(x)).
\end{align*}
Now for an arbitrary $i\in I$, $\boldsymbol{u}_i(t)$ is simply the solution to the continuum voter model (\ref{PDE}) with connected kernel $\widehat{W}_i$ and initial condition function $\boldsymbol{g}^{\phi_i^{-1}}$ (see within the proof of \Cref{decomposedynamics}). 

Thus, \Cref{consisconstant} shows that 
 \begin{align*}
u^*(x) &=\sum_{i\in I} \mathbbm{1}_{J_i}(x)\,\Big( \int_0^1 g(\phi_i^{-1}(y))\, dy \Big)\\
\int_0^1 g(z)\, dz &=\sum_{i\in I} \mathbbm{1}_{J_i}(x)\,\Big( \lambda(J_i)\, \int_{J_i} g(z)\, dz \Big)&&\text{substituting $\phi_i^{-1}(y)$ with $z$}.
\end{align*}
Hence, to avoid a contradiction the necessary condition (\ref{neccon}) must be satisfied.
 \end{proof}

%%%%%%%%%%%%%%%%%%%%%%%%%%%%%%%%%%%%%%%%%%%%%%%%%%%%

\section{Guaranteeing consensus}\label{s6.1}

%%%%%%%%%%%%%%%%%%%%%%%%%%%%%%%%%%%%%%%%%%%%%%%%%%%%

In \Cref{s6} it was shown that the continuum voting model (\ref{PDE}) can be used to approximately solve the consensus problem for the finite voting model (\ref{IVP}) on large graphs (see \Cref{keytheorem}). However, the applicability of \Cref{keytheorem} crucially depends on understanding which graph limits, or kernels, attain consensus in the continuum model (\ref{PDE}). This section introduces two new notions called \emph{twin-sets} and \emph{twin-kernels} which will allow a broad class of graphons to be identified which the guarantee consensus (see \Cref{twingraphthm}).

We now introduce the notion of a \emph{twin-set}, a \emph{maximal twin-set} and a \emph{twin-kernel}. Two sets $A$ and $B$ will be called equal if 
\begin{align}\label{seteq}
\lambda\big\{A\, \Delta \,B\big\}=0,
\end{align}
where $\lambda$ denotes the Lebesgue measure and $\Delta$ denotes the symmetric difference. For notational convenience we will denote set equality simply as $A=B$.

\begin{definition}\label{twindef} 
Let $A \subseteq [0,1]$ with nonzero measure and let $W$ be a kernel.\linebreak We say that $A$ is a \em twin-set \em of $W$ if there exists a function $a: A\times A \rightarrow \mathbb{R}$ such that 
\begin{align}\label{twindefeq}
W(x,y)=a(x, x')\,W(x',y) \qquad\text{for all $x,x'\in A$ and almost every $y\in [0,1]$.}
\end{align}
We say a twin-set is \emph{maximal} if for any twin-set $B$ such that
$$A\subseteq B$$
then $B=A$. If $W$ has a finite number of maximal twin-sets $\{A_i\}_{i=1}^n$ such that 
$$[0,\,1] = \bigcup_{i=1}^n A_i$$ where set equality is understood as in (\ref{seteq}), then we say that $W$ is a \em twin-kernel\em. If in addition $W$ is a graphon we say that $W$ is a \em twin-graphon\em. 
\end{definition}

To illustrate the notions introduced above the following example is provided.

\begin{Example}
Every step kernel is a twin-kernel. To see this simply partition $[0,1]$ into sets $\{A_i\}_{i=1}^n$ and define $a_i(x,x')=1$ for all $x,x'\in A_i$ for each $i=1,\, 2, \ldots, n$. Further, since every weighted graph can be represented as a step kernel it follows that every weighted graph is also a twin-kernel. 

More generally, every kernel $W$ of the form
$$W(x,y)=f(x)\, f(y) \qquad \text{for $x,y\in [0,1]$} $$
is a twin-kernel. To see this we define the sets $A_1=\{x\in [0,1]: f(x)\neq 0\}$ and $A_2=[0,1]/ A_1$ then
\begin{align*}
a_1(x,x')&=0  && \text{for all $x,x'\in A_1$; and,}\\
a_2(x,x')&=\frac{f(x)}{f(x')}&& \text{for all $x,x'\in A_2$.}
\end{align*}
This also implies that a twin-kernel need not be connected.

An example of a kernel which is not a twin-kernel is the following 
$$W(x,y)=\frac{1}{2}(x+y),$$
which does not have any twin-sets and so is not a twin-kernel. \hfill $\diamond$
\end{Example}

From a combinatorial perspective, twin-sets can be thought of as a generalisation of `blow-up' graphs defined below.  

\begin{definition} (Adapted from \cite[Section 3.3]{Lov}) \label{blowup} \\
Let $r$ be a positive integer, the \em $r$-blow-up graph \em $G(r)$ of a weighted graph $G$ is obtained by replacing each vertex of $G$ by $r$ copies such that the edge weight between two vertices is equal to the edge weight between the original vertices. 
\end{definition}

Twin-kernels generalise the above definition for finite graphs in two ways. Firstly, by allowing different vertices in the graph $G$ to be replaced by a number $r_v$ of copies which depends on the vertex $v\in V(G)$. Secondly, by allowing the edge weight between copies of a given vertex to vary in consistent way as defined by the function $a(x, x')$ (\ref{twindefeq}). 

Both of these generalisation are illustrated below. The graph on the right has two copies of vertex $v_1$ whilst vertices $v_2$ and $v_3$ are not replicated. The edge weights of $v_1'$ between a given vertex, say $v_{-i}$, is always three times the edge weight between $v_1$ and $v_{-i}$. Vertices within twin-sets have additional structure for example the total degree of $v_1'$ is three times that of $v_1$.

\begin{center}
 \begin{tikzpicture}[scale= 0.7]

\draw [-] (0,0)--(1.25,2)--(2.5,0);

\draw [-] (7.5,0)--(9,2);
\draw [-] (12,0)--(10.5, 2)--(7.5,0);
\draw [-] (12,0)--(9,2);

\node [left] at (0.5,1.25) {$1$};
\node [right] at (1.8,1.25) {$2$};

\node [left] at (8,1.25) {$1$};
\node [right] at (11.5,1.25) {$6$};
\node [right] at (8.75,0.7) {$3$};
\node [left] at (10.75,0.7) {$2$};

\draw [fill=white] (0,0) circle (0.35);
\node [] at (0, 0) {$v_2$};

\draw [fill=white] (2.5,0) circle (0.35);
\node [] at (2.5, 0) {$v_3$};

\draw [fill=white] (1.25,2) circle (0.35);
\node [] at (1.25,2) {$v_1$};

\draw [fill=white] (7.5,0) circle (0.35);
\node [] at (7.5, 0) {$v_2$};

\draw [fill=white] (12,0) circle (0.35);
\node [] at (12, 0) {$v_3$};

\draw [fill=white] (9,2) circle (0.35);
\node [] at (9,2) {$v_1$};
\draw [fill=white] (10.5,2) circle (0.35);
\node [] at (10.5,2) {$v_1'$};

\end{tikzpicture} 
\end{center}

We now preset two additional properties of twin-sets. 

\begin{Proposition}\label{rem1}
Let $W$ be a kernel with maximal twin-sets $A$ and $B$ such that $A\neq B$ then $$A\cap B=\emptyset.$$
\end{Proposition}

\begin{proof}
Let $A$ and $B$ be maximal twin-sets and define the associated function $a(\cdot\, ,\, \cdot)$ in (\ref{twindefeq}) by  $a_A(\cdot \,,\, \cdot)$ and $a_B(\cdot \,,\, \cdot)$, respectively. For the purpose of a contradiction suppose that $A\cap B \neq \emptyset$. Then we see that $A\cup B$ is also a twin-set by first choosing an arbitrary $z\in A\cap B$ and defining
\begin{align*}
a(x, x')=\begin{cases}
a_A(x,x') &\text{if $x, x'\in A$,}\\
a_B(x,x') &\text{if $x, x'\in B$,}\\
a_A(x,z)\, a_B(z,x')&\text{ if $x\in  A/ B$ and $x'\in B/ A$}.
\end{cases}
\end{align*}
But this means that we have a twin-set $A\cup B$ which contains the maximal twin-sets $A$ and $B$. This contradicts the maximality of $A$ and $B$ and so it must be the case that $A\cap B=\emptyset$.
\end{proof}

Recall the normalised degree function, $d_W$ (\ref{normdeg}). 

\begin{Proposition}\label{rem2}
Let $W$ be a kernel with twin-set $A$ then 
\begin{align}\label{part1}
d_W(x)=0 \qquad \text{for all $x\in A$, or }\qquad d_W(x)\neq 0 \qquad \text{for all $x\in A$}.
\end{align}
In the latter case for any pair $x,x'\in A$ we have
$$a(x,x')=\frac{d_W(x)}{d_W(x')}.$$
\end{Proposition}

\begin{proof}
First note that for any pair $x, x' \in A$ we have 
\begin{align}\label{transitivenorm}
d_W(x)&=\int_0^1 W(x,y) \, dy\notag\\
&=\int_0^1a(x,x') W(x',y) \, dy &&\text{by (\ref{twindefeq}),}\notag\\
&=a(x,x') d_W(x').
\end{align}
Now if $0\in \{ a(x,x') : x,x'\in A\}$ we claim that
\begin{align}\label{extraclaim}
 d_W(x)=0\qquad \text{for all $x\in A$.}
 \end{align}

Suppose that $0\in \{ a(x,x') : x,x'\in A\}$; that is, there exists a pair $x, x'\in A$ such that $a(x,x')=0$. From (\ref{twindefeq}) it then follows that 
$$W(x,y)=a(x,x')\, W(x',y)=0 \qquad \text{for almost every $y\in [0,1]$,}$$
and so $d_W(x)=0$. Now consider an arbitrary $z\in A$, from (\ref{transitivenorm}) we see that 
\begin{align}\label{extraeq1}
d_W(z)&=a(z,x) d_W(x)=0.
\end{align}
Since $z\in A$ was chosen arbitrarily (\ref{extraeq1}) extends to all $z\in A$ and proves the claim (\ref{extraclaim}).

Now suppose that $0\notin\{ a(x,x') : x,x'\in A\}$ then for a particular $x\in A$ either $d_W(x)=0$ or $d_W(x)\neq 0$. If $d_W(x)=0$ then for any $x'\in A$
$$d_W(x')=a(x',x)\, d_W(x)=0.$$
Otherwise, $d_W(x)\neq 0$ and 
$$d_W(x')=a(x',x)\, d_W(x) \neq 0,$$
since $0\notin  \{ a(x,x') : x,x'\in A\}$ - completing the proof of (\ref{part1}). The final claim follows by rearranging (\ref{transitivenorm}) when $d_W(x)\neq 0$ for all $x\in A$.
\end{proof}

We will now exploit the additional structure provided by twin-sets to provide an insight into the consensus problem.

\begin{Lemma}
\label{prevlem}
Let $W\in \mathcal{W}_1$ with a twin-set $A$ and let $\boldsymbol{g}\in L_\infty([0,1])$. If $\boldsymbol{u}$ is a bounded solution to the IVP (\ref{PDE}) and $d_W(x)\neq 0$ for all $x\in A$ then
\begin{align}
\label{eq1lem0}
u^*(x)&=C \qquad\text{for almost every $x\in A$,}
\end{align}
where $C$ is some constant. That is, $\boldsymbol{u}$ reaches consensus on $A$.
\end{Lemma}

\begin{proof}
By assumption the solution $\boldsymbol{u}$ is bounded. This ensures $\lim_{t\rightarrow \infty} u(x,t)$ exists for almost every $x\in [0,\,1]$ by \Cref{ctsmart}. Now fix an $x\in A$ such that
 $$u^*(x)=\lim_{t\rightarrow \infty} u(x, t),$$
  note that this condition holds almost everywhere in $A$. Then by the Dominated Convergence Theorem \cite[Theorem 2.24]{Foll}
\begin{align}
\label{stabeq}
\lim_{\tau \rightarrow \infty} \frac{\partial u(x,t)}{\partial t}\Big|^{t=\tau}&=\lim_{\tau \rightarrow \infty} \int_0^1 W(x,y)\,\big(u(y,\tau)-u(x,\tau)\big)\,dy\notag\\
& = \int_0^1\lim_{\tau \rightarrow \infty} \Big(W(x,y)\,\big(u(y,\tau)-u(x,\tau)\big)\Big)\,dy\notag\\
&= \int_0^1 W(x,y)\,\big(u^*(y)-u^*(x)\big)\,dy.
\end{align}
That is, the limit on the left hand side exists. So for this fixed $x\in A$ there exists a finite constant $L$ such that 
\begin{align*}
\lim_{\tau \rightarrow \infty}\big(u(x,\tau)+u_t(x, \tau)\big)=L,
\end{align*}
where $u_t(x, \tau)$ denotes partial derivative of $u(x,t)$ with respect to $t$ evaluated at $(x,\tau)$; that is
$$u_t(x,\tau)=\frac{\partial u(x,t)}{\partial t} \bigg|^{t=\tau}.$$
This implies that $ \lim_{\tau \rightarrow \infty}u_t(x,\tau)=0$, since L'H$\hat{\mathrm{o}}$pital's rule gives
$$ \lim_{\tau\rightarrow \infty} u(x,\tau)=\lim_{\tau \rightarrow \infty} \frac{e^\tau u(x,\tau)}{e^\tau}=\lim_{\tau \rightarrow \infty} \frac{e^\tau \big(u(x,\tau)+u_t(x,\tau)\big)}{e^\tau}=L,$$
note that $\frac{\partial \, u(x,t)}{\partial t}$ is defined by the IVP (\ref{PDE}) and so for fixed $x$ the solutions $u(x,t)$ are differentiable with respect to $t$. Now rearranging (\ref{stabeq}), gives the following `stabilising condition' for this fixed $x\in A$
\begin{align}
\label{stabeq1}
d_W(x)\,u^*(x)&=\int_0^1W(x,y)\,u^*(y)\,dy.
\end{align}

Since $d_W(x)\neq 0$ by assumption, we can divide by $d_W(x)$ which gives
\begin{align}
\label{twinstabcond}
u^*(x)&=\frac{1}{d_W(x)}\int_0^1W(x,y)\,u^*(y)\,dy.
\end{align}
Let $x'\in A$ such that $u^*(x')$ exists. Then $d_W(x)=a(x, x')\,d_W(x')$ (recall (\ref{transitivenorm})) and 
\begin{align*}
\int_0^1W(x,y)\,u^*(y)\,dy=a(x, x')\int_0^1W(x',y)\,u^*(y)\,dy.
\end{align*}
Thus the right hand side of (\ref{twinstabcond}) is invariant for all elements in the same twin set, since the $a(x, x')$ term in the numerator and denominator cancels. Hence, 
\begin{align*}
u^*(x)&=u^*(x') \qquad\text{ for almost every $x, x'\in A$.}
\end{align*}
That is, $\boldsymbol{u}^*$ is almost everywhere constant on $A$. This proves (\ref{eq1lem0}).
\end{proof}

By restricting the above lemma to the set of graphons $\mathcal{W}_0$ we attain a more elegant result.

\begin{Lemma}
\label{prevlemgraphon}
Let $W\in \mathcal{W}_0$ with a twin-set $A$ and let $\boldsymbol{g}\in L_\infty([0,1])$. If $\boldsymbol{u}$ is a solution to the IVP (\ref{PDE}) then we have
\begin{align}
\label{eq1lem}
u^*(x)&=g(x) &&\text{for almost every $x\in A$; or,}\\
\label{eq2lem}
u^*(x)&=C &&\text{for almost every $x\in A$,}
\end{align}
where $C$ is some constant. In the latter case, $\boldsymbol{u}$ reaches consensus on $A$ (recall \Cref{condef}).
\end{Lemma}

\begin{proof} Let $\boldsymbol{u}$ be the solution the IVP (\ref{PDE}) with graphon $W$ and initial condition $\boldsymbol{g}$. First note that \Cref{rem2} states that either $d_W(x)=0$ for all $x\in A$, or $d_W(x)\neq 0$ for all $x\in A$. In the first case, since $W\in \mathcal{W}_0$, this implies that for all $x\in A$ we have $W(x,y)=0$ for almost every $y\in [0,1]$. Thus, the solution of the IVP is $u(x,t)=g(x)$ for all $x\in A$. In the second case, we can apply \Cref{prevlem} since $\boldsymbol{u}$ is automatically a bounded solution (see \Cref{graphonboundedsol}). This completes the proof.
\end{proof}

In fact the above lemma extends to twin-graphons via the following theorem.

\begin{Theorem}\label{twingraphthm}
Let $W\in \mathcal{W}_0$ be a twin-graphon with $n$ connected components and associated sets denoted by $\{A_i\}_{i=1}^n$ (see \Cref{directsum} and \Cref{decompker}), and let $\boldsymbol{g}\in L_\infty([0,1])$. If $\boldsymbol{u}$ solves the IVP (\ref{PDE}) then  
$$u^*(x)=\sum_{i=1}^n \mathbbm{1}_{A_i} (x) \Bigg(\lambda(J_i)\,\int_{A_i} g(y)\, dy\Bigg)\qquad \text{ almost everywhere,}$$
where $\lambda(J_i)$ denotes the Lebesgue measure of the set $J_i$. It follows that consensus is reached if $W$ is connected or $\lambda(J_i)\,  \int_{A_i} g(y)\, dy$ is constant for all $i\in [n]$.
\end{Theorem}

\begin{proof}
First note that the solutions to the continuum voter model on $W\in \mathcal{W}_0$ can be decomposed into solutions on connected graphons (see \Cref{decomposedynamics}). Thus, we will prove the result for the solution $\boldsymbol{\widehat{u}}$ to the IVP (\ref{PDE}) with connected graphon, say $\widehat{W}$, and then extend the result to complete the theorem.

Recall that $\widehat{W}\in \mathcal{W}_0$ and so by \Cref{graphonboundedsol}, $\boldsymbol{\widehat{u}}$ is a bounded solution. This guarantees the existence of $\lim_{t\rightarrow \infty} \boldsymbol{\widehat{u}}(t)$ almost everywhere (see \Cref{ctsmart}). Also $\widehat{W}$ is a twin-graphon and so there exists a finite family of maximal twin-sets $\{\widehat{A}_i\}_{i=1}^n$, which are disjoint (see \Cref{rem1}). Thus, $[0,1]=\dot\cup_{i=1}^n \widehat{A}_i$ where set equality is understood as in (\ref{seteq}).

Now since $\widehat{W}\in \mathcal{W}_0$ is connected, $d_{\widehat{W}}(x)\neq 0$ for almost every $x\in [0,1]$. Thus, \Cref{prevlem} states that $\boldsymbol{u}^*$ is almost everywhere constant on the twin-sets and so we have
\begin{align*}
\widehat{u}^*(x)=\sum_{i=1}^{n} c_i \mathbbm{1}_{\widehat{A}_i}(x) \qquad\text{for some positive integer $n$ and for some constant $c_i$},
\end{align*}
where $\mathbbm{1}_{A}$ denotes the indicator function with respect to the set $A$. We can assume that $c_i\neq c_j$ for $i\neq j$, by taking set unions if this is not the case.

 Suppose for the purpose of a contradiction that $\boldsymbol{\widehat{u}}^*$ is not almost everywhere constant. Then there exists $i\in \{1,\,2,\ldots, n\}$ such that $c_i >c_j$ for all $j\neq i$. Since $W$ is connected, there must exists a set $S\subseteq A_i$ of positive measure such that
\begin{align}
\label{conneq1}
\int_{A_i^c} \widehat{W}(x,y)\,dy> 0 \qquad\text{for all } x\in S.
\end{align}
Now fix $x\in S$. The stabilising condition (\ref{stabeq1}) then gives
\begin{align*}
d_{\widehat{W}}(x)\,c_i&=\sum_{j=1}^{n} c_j \int_{A_j} \widehat{W}(x,y)\,dy\\
&<c_i \,\sum_{j=1}^{n}\int_{A_j} \widehat{W}(x,y)\,dy && \text{ since $c_i>c_j$ for all $j\neq i$ and by (\ref{conneq1})}\\
&=d_{\widehat{W}}(x)\,c_i,
\end{align*}
which is a contradiction. Thus it must be the case that $\boldsymbol{\widehat{u}}^*$ is constant almost everywhere. If $\boldsymbol{\widehat{u}}^*$ is almost everywhere constant then $\boldsymbol{\widehat{u}}$ reaches consensus on $[0,1]$. Applying \Cref{consisconstant} then gives the required result. 
\end{proof}

The nonnegativity of the kernel (i.e. $W\in \mathcal{W}_0$) in \Cref{twingraphthm} is essential. The following example shows that a connected, twin-kernel which takes negative values can have a non-constant limit: that is, consensus is never attained.

\begin{Example}
Consider the weighted graph $G$ illustrated below 
\begin{center}
 \begin{tikzpicture}

\node [left] at (-0,1) {$-1$};
\node [right] at (2,1) {$-1$};
\node [above] at (1,2) {$1$};
\node [below] at (1,0) {$1$};

\draw [-] (0,0)--(2,0)--(2,2)--(0,2)--cycle;

\draw [fill=white] (0,0) circle (0.3);
\node [] at (0,0) {$v_3$};
\draw [fill=white] (0,2) circle (0.3);
\node [] at (0,2) {$v_1$};
\draw [fill=white] (2,2) circle (0.3);
\node [] at (2,2) {$v_2$};
\draw [fill=white] (2,0) circle (0.3);
\node [] at (2,0) {$v_4$};
\end{tikzpicture} 
\end{center}

The kernel representing the graph $W_G$ (see (\ref{pixel})) is a connected twin-kernel. However, it takes negative values and so \Cref{twingraphthm} does not apply since $W_G\notin \mathcal{W}_0$. Let $I_i=[\frac{i-1}{4}, \frac{i}{4}]$ for $i=1,2,3,4$. Then any initial condition function $\boldsymbol{g}$ such that 
\begin{align}
\label{weirdcond}
\int_{I_i} g(y)dy=\int_{I_j}g(y)dy \qquad\text{for $(i,j)=(2,3)$ or $(i,j)=(1,4)$ }
\end{align}
will force $\boldsymbol{u}(t)=\boldsymbol{g}=\boldsymbol{u}^*$ for all $t\in \mathbb{R}^{\ge 0}$. However, $\boldsymbol{g}$ need not be constant, and hence $\boldsymbol{u}(t)$ need not reach consensus. For example, the function 
\begin{align*}
g(x)&=\begin{cases}
-1+4x &\text{ if } 0\le x\le \frac{1}{2},\\
3-4x &\text{ if } \frac{1}{2}< x\le 1,
\end{cases}
\end{align*}
 will satisfy (\ref{weirdcond}).\hfill $\diamond$
\end{Example}

%%%%%%%%%%%%%%%%%%%%%%%%%%%%%%%%%%%%

\section{Extension: finite voting model with random weights}\label{srand}

%%%%%%%%%%%%%%%%%%%%%%%%%%%%%%%%%%%%

So far in this paper we have only considered the voter model on a deterministic graph. In this section we extend our analysis to include random simple graphs. In particular, we are able to formulate a probabilistic version of \Cref{keytheorem} which shows that the continuum voter model (\ref{PDE}) can be used to approximately solve the finite voter model (\ref{IVP}) when the underlying graph is random. This extension allows our results to be applied to many well-known random graph processes such as the Watts-Strogatz small world graph.

First we define $W$-random graphs which will be used to generate the random graph model for the voter model (\ref{IVP}).

\begin{definition}
Let $n$ be a positive integer and let $W\in \mathcal{W}_0$ be a graphon. A graph $G_n=(V_n, \, E_n)$ is called a $W$-random graph if $V_n=[n]$ and for every $(i, \, j)\in [n]^2$ such that $i \neq j$ we have
$$\mathbb{P}\Big[(i,\, j)\in E_n\Big]=W\Bigg(\frac{i}{n}, \, \frac{j}{n}\Bigg),$$
where each decision whether to include $(i\, j)\in E_n$ is made independently. A $W$-random graph on vertex set $[n]$ and graphon $W$ is denoted by $\mathbb{G}_n(W)$.
\end{definition}

Let $p\in [0,1]$ if $W(x,y)=p$ for all $(x,y)\in [0,1]^2$ then $\mathbb{G}_n(W)$ is the well-known Erd{\H o}s-R\'enyi graph often denoted by $G(n,p)$. Letting the value of $W$ vary over $[0,1]^2$ provides a significantly richer set of random graphs. For example, a generalisation of the Watts-Strogatz small worlds graph is given by $\mathbb{G}_n(W_p)$ where 
\begin{align}\label{wattsgraphon}
W_p(x,y)&=(1-p)\, W(x,\,y)+p\,\big(1-W(x,\,y)\big)&& \text{for some } p\in [0,\, 0.5].
\end{align}

We consider the voting model introduced in \Cref{s2} where the underlying graph is random and given by a $W$-random graph, $\mathbb{G}_n(W)$. That is, the voting model is defined by
\begin{align}
\label{randIVP}
\begin{cases}
\frac{du_i^{(n)}(t)}{dt}&=\frac{1}{n}\sum_{j=1}^n \beta_{ij}^{(n)}\Big(u_j^{(n)}(t)-u_i^{(n)}(t)\Big)\qquad \text{for all }\, i\in [n] \text{ and } t\in \mathbb{R}^{>0},\\
\boldsymbol{u}^{(n)}(0)&=\boldsymbol{g}^{(n)},
\end{cases}
\end{align}
where $\beta_{ij}^{(n)}=1$ if and only if $(i, j)\in E(\mathbb{G}_n(W))$. The vector $\boldsymbol{g}^{(n)}\in \mathbb{R}^n$ is defined in the same way as the deterministic case (\ref{gsubn}).

It was shown in \cite{MedWrand} that finite processes such as the voting model above, under certain conditions, can be approximated by the continuum model 
\begin{align*}
\begin{cases}
\frac{\partial u(x,t)}{\partial t}& = \int_0^1 W(x,y)\Big(u(y,t)-u(x,t)\Big)dy \qquad \text{for all } x\in [0,1] \text{ and }  t\in \mathbb{R}^{>0},\\
\boldsymbol{u}(0)&=\boldsymbol{g},
\end{cases}
\end{align*}
where $W$ represents the graph limit of the $W$-random graph sequence $\{\mathbb{G}_n(W)\}_{n=1}^\infty$. Fortunately, we have the following result which defines the graph limit of a sequence of $W$-random graphs when $W$ is continuous almost everywhere.

\begin{Lemma} \cite[Lemma 2.5]{Borgsrandgrown}\\
If $W\in \mathcal{W}_0$ is continuous on $[0,\, 1]^2$ almost everywhere, then the sequence $\{\mathbb{G}_n(W)\}$ converges almost surely with the limit given by the graphon $W$.
\end{Lemma}

We can now formally define a sufficient condition for the continuum limit to approximate the finite voter model on large graphs. The theorem follows as an application of \cite[Theorem 4.3]{MedWrand}.

\begin{Theorem}\label{convinprob}
Let $W\in \mathcal{W}_0$ be a graphon such that $W$ is almost everywhere continuous on $[0,\, 1]^2$ and let $\boldsymbol{g}\in L_\infty[0,\, 1]$. Let $\boldsymbol{u}_n$ and $\boldsymbol{u}$ be solutions to IVP (\ref{PDE}) and (\ref{randIVP}) respectively. If 
\begin{align}\label{randcond}
\min_{t\in [0, \, T]} \int_{[0, \, 1]^2} \Big(u(y,\,t)-u(x,\, t)\Big)\,W(x,\, y)\, \Big(1-W(x, \, y) \Big)\, dx\, dy>0
\end{align}
for some $T>0$ then 
$$\|\boldsymbol{u}_n-\boldsymbol{u}\|_{C([0, \,T]; L_2([0,\,1]))}\overset{p}{\to} 0\qquad \text{as} \qquad n\rightarrow\infty.$$
The convergence above is in probability.
\end{Theorem}

Our key result in the deterministic setting (\Cref{keytheorem}) easily extends to the probabilistic setting with the added condition (\ref{randcond}). First we define a probabilistic notion of convergence.

\begin{definition}\cite[Section 1.2]{Jan2}\\
Let $A_n$ be an event describing a property of a random structure depending on a parameter $n$. We say that $A_n$ holds \emph{asymptotically almost surely} (or \emph{with high probability}) if 
$$\mathbb{P}[A_n] \rightarrow 1 \qquad \text{ as }\quad n\rightarrow \infty.$$
\end{definition}

\begin{Theorem} \label{keytheorem1}
Let $\boldsymbol{u}$ be a solution to the IVP (\ref{PDE}), with graphon $W\in\mathcal{W}_0$ and initial condition function $\boldsymbol{g}\in L_\infty([0,1])$ such that (\ref{randcond}) holds. Denote the solution to the $n$-th approximate IVP (\ref{randIVP}) on the W-random graph, $\mathbb{G}_n(W)$, by $\boldsymbol{u}_n$ for each positive integer $n$. 

Suppose that $\boldsymbol{u}$ reaches consensus on $[0,1]$ and let $D$ be any positive real number. Then for every $\varepsilon>0$ and for every $c>0$, there exists $T=T(\varepsilon)$ and a subset $S_t\subseteq[0,1]^2$ with $\lambda(S_t)<c^2$ such that asymptotically almost surely
\begin{align}
\label{res1}
|u_n(x,t)-u_n(y,t)|\le \varepsilon \qquad \text{ for all $(x,y)\in [0,1]^2\setminus S_t$ and $t\in [T,T+D]$.}
\end{align}
\end{Theorem}

\begin{proof}
The proof follows the proof of the deterministic analog \Cref{keytheorem} by noting that the convergence in (\ref{bound3}) occurs in probability by \Cref{convinprob}. Letting the event $A_n$ be (\ref{res1}) we see that the statement holds asymptotically almost surely.
\end{proof}

\begin{Example}
Let $p \in (0, \, 0.5)$ and let $W\in \mathcal{W}_0$ be a graphon. We consider the voting model on a sequence of Watts-Strogatz small world graphs $\mathbb{G}_n(W_p)$ (defined in (\ref{wattsgraphon})) for some initial condition $\boldsymbol{g}$ such that $\int \boldsymbol{g}=0$.

If (\ref{randcond}) holds i.e. 
\begin{align}
\min_{t\in [0, \, T]} \int_{[0, \, 1]^2} \Big(u(y,\,t)-u(x,\, t)\Big)\,W(x,\, y)\, \Big(1-W(x, \, y) \Big)\, dx\, dy>0\notag\\
\iff \min_{t\in [0, \, T]} \int_{[0, \, 1]^2} \Big(W(x,y)-W(x, y)^2 \Big)\,\Big(u(y,\,t)-u(x,\, t)\Big)\, dx\, dy>0,\label{equivcondrand}
\end{align}
then by \Cref{keytheorem1} we know that if consensus is attained on $W_p$ then the finite voter model will asymptotically almost surely be close to attaining consensus. Thus, we turn our focus towards the continuum model
\begin{align*}
\frac{\partial u(x,t)}{\partial t}& = \int_0^1 W_p(x,y)\Big(u(y,t)-u(x,t)\Big)dy. 
\end{align*}

Simplifying we see that 
\begin{align*}
\frac{\partial u(x,t)}{\partial t}& = \int_0^1 W_p(x,y)\Big(u(y,t)-u(x,t)\Big)dy\\
&= \int_0^1 (1-2p)\,W(x,\,y)\,\Big(u(y,t)-u(x,t)\Big)dy+p\int_0^1 u(y,t)\, dy\\
&\qquad-p \,u(x,t)\\
&= \int_0^1 (1-2p)\,W(x,\,y)\,\Big(u(y,t)-u(x,t)\Big)dy+p\int_0^1 g(y)\, dy\\
&\qquad-p \,u(x,t)\\
&= \int_0^1 (1-2p)\,W(x,\,y)\,\Big(u(y,t)-u(x,t)\Big)dy-p \,u(x,t),
\end{align*}
the final two equalities follow from \Cref{conservation} and the assumption that $\int \boldsymbol{g}=0$. Thus rearranging and using an integrating factor of $e^{p\, t}$ we see that the solution must satisfy
\begin{align*}
\frac{\partial e^{p\,t}\,u(x,t)}{\partial t}&=\int_0^1 (1-2p)\, W(x,y)\, \big(e^{p\,t}u(y,t)-e^{p\, t}\,u(x,t)\big)\, dy.
\end{align*}
Let $v(x,t)$ be the unique solution to the IVP on graphon $(1-2p) W(x,y)$ with initial condition $\boldsymbol{g}$, then it follows by the uniqueness property (\Cref{existanduniq}) that
$$u(x,t)=e^{-p\, t} \, v(x,t).$$
It is clear that if $v(x,t)$ attains consensus then this is sufficient to show that $u(x,t)$ also attains consensus. 

We conclude that if $W$ is a connected twin-graphon and (\ref{equivcondrand}) holds then $\boldsymbol{v}$ and hence $\boldsymbol{u}$ both reach consensus in the small worlds voter model. Applying \Cref{keytheorem1} then implies that the finite voter model will asymptotically almost surely be close to attaining consensus.
\hfill $\diamond$
\end{Example}

%For acknowledgements section, please don't number the section, please begin it with \section*{Acknowledgements}
\section*{Acknowledgments} I would like to acknowledge Catherine Greenhill's and Richard Holden's assistance in developing various aspects of this paper. I would also like to thank Georgi Medvedev, Oleg Pikhurko and the referee for their detailed and constructive feedback.

% You may incorporate your references as follows in your main tex file.
% Using BibTex is not recommended but can be handled.

\medskip
% The data information below will be filled by AIMS editorial staff
%Received May 2016; revised October 2017.
\medskip

\end{document}